\documentclass[12pt,twoside]{amsart}

\usepackage{amssymb,amsmath}
\theoremstyle{plain}
\newtheorem{theorem}{Theorem}[section]
\newtheorem{lemma}[theorem]{Lemma}

\newtheorem{conjecture}[theorem]{Conjecture}

\theoremstyle{definition}
\newtheorem{remark}[theorem]{Remark}
\newtheorem{defn}[theorem]{Definition}

\numberwithin{equation}{section}

  \input xypic
  \xyoption{all}

 \usepackage{tikz}
\usetikzlibrary{arrows}

 \usepackage{times}

  \usepackage{amsmath, amsthm, amsfonts}

\usepackage{amssymb}
\usepackage{amsmath}
\usepackage{epsfig}
\usepackage{amsthm}
\usepackage{euscript}
\usepackage{latexsym, mathrsfs}

\numberwithin{equation}{section}

\def \F{\mathcal F}
\def \M{W}

\def \mc{\mathcal}

\def \v{\vskip 0.1in}
\def \n{\noindent}

\def \real{\mathbb{R}}

\def \p{\partial}

\def\t{\mathfrak{t}}

\def \M{{\mathcal M}}
 
\def \F{{\mathcal F}}

\def \p{\partial}
\def \F {\mathcal F}
\def \H {\mathcal H}

\def \M {\mathcal M}
\def \E {\mathcal E}
\def \varphi {\phi}

\begin{document}

\title{Extremal metrics on toric manifolds and homogeneous toric bundles}
\author[Li]{An-Min Li}
\address{Department of Mathematics,
Sichuan University,
 Chengdu, 610064, China}
\email{anminliscu@126.com}
\author[Lian]{Zhao Lian}
\address{School of Mathematics,
Southwest Jiaotong University,
Chengdu, 611756, China}
\email{zhaolian@swjtu.edu.cn}
\author[Sheng]{Li Sheng}
\address{Department of Mathematics,
Sichuan University,
Chengdu, 610064, China}
\email[Corresponding author]{lshengscu@gmail.com}
\thanks{ Li acknowledges the support of NSFC Grant NSFC11890660, NSFC11821001, NSFC11890663. \\
${}\quad\ $Lian acknowledges the support of NSFC Grant NSFC11901480, NSFC12071322, NSFC12071059.\\
${}\quad\ $Sheng acknowledges the support of NSFC Grant NSFC11871352, NSFC1196131001.}
{\abstract
In this paper, we prove the Yau-Tian-Donaldson conjecture of the filtration version for toric manifolds and homogeneous toric bundles.\endabstract
}

\maketitle


\section{\bf Introduction }\label{Sec-Intro}

Extremal metrics, introduced by E. Calabi, have been studied intensively in the past 30 years. The necessary conditions for the existence are conjectured to be related to certain stabilities. There are many works on this aspect. In particular,  Tian first made great progress toward understanding this. It was Tian who first  gave an analytic stability condition which he proved that is equivalent to the existence of a K\"{a}hler-Einstein metric (\cite{T1}). This condition is the {\em properness} of the Mabuchi functional. In \cite{T1}, Tian also defined the algebro-geometric notion of K-stability. In \cite{D1}, Donaldson generalized Tian's definition of K-stability by giving an algebro-geometric definition of the Futaki invariant and conjectured that it is equivalent to the existence of a cscK metric. The conjecture can be formulated as
\begin{conjecture}[Yau\cite{Y}, Tian\cite{T1}, Donaldson\cite{D1}]\label{conj_1}
The manifold $M$ admits a cscK metric in the class $c_1(L)$ if and only if $(M,L)$ is K-stable.
\end{conjecture}
This conjecture were generalized to extremal metrics. Recently, there have been many progresses towards this conjecture (see \cite{BDL2}, \cite{CC1}, \cite{CC2}, \cite{DL}, \cite{Li}).
\v
In a sequence of papers, Donaldson initiated  a program to study the extremal metrics on toric manifolds:   Donaldson formulated  $K$-stability for polytopes and conjectured  that the stability implies the existence of the cscK metric on toric manifolds. In a sequence of papers (cf. \cite{D2,D3,D4}), Donaldson solved the problem for cscK metrics on toric surfaces. In \cite{CLS2}, Chen-Li-Sheng  solved the problem for extremal metrics for toric surfaces.

The conjecture is still open in general, and examples in \cite{ACGT} indicate that the condition of $K$-stability may be modified for polarized manifolds (see \cite{Sz1}, \cite{BHJ}, \cite{His}). Recently, the existence of cscK metrics under uniform stability condition has been proved for toric manifolds (see \cite{CC2}, \cite{Li}).

Sz\'ekelyhidi (see \cite{Sz1}) introduced $\widehat{K}$-polystability and proved that $(M, L)$ is $\widehat{K}$-polystable when $M$ admits a cscK metric in $c_1(L)$ and has discrete automorphism group.
He  states a variant of the Yau-Tian-Donaldson conjecture.
\begin{conjecture}
The manifold $M$ admits an extremal metric in $c_{1}(L)$ if and only if $(M, L)$ is relatively $\widehat{K}$-stable.
\end{conjecture}

In this paper, we consider toric manifolds and homogeneous toric bundles. The equation for toric manifolds is (\ref{eqn-1.3}). We prove the following theorem.
\begin{theorem}\label{theorem_1.3}
Let $(M,\omega)$ be a $n$-dimensional compact toric manifold and $\Delta$ be its Delzant polytope. $A$ is a smooth function on $\bar\Delta$. Then $(\Delta, A)$ is relatively $\widehat K$-polystable if and only if there is a smooth $T^{n}$-invariant metric $g$ on $M$ such that the scalar curvature of $g$ is $A$.
\end{theorem}

Let $A$ be a constant or a linear function on $\Delta$. As a consequence of Theorem \ref{theorem_1.3} we have solved the Yau-Tian-Donaldson conjecture of the filtration version for $n$-dimensional toric manifolds.


In \cite{D5}, Donaldson suggested to study the existence of K\"ahler metrics
of constant scalar curvatures on toric fibrations. He presented
the underlying equation, which we call the generalized Abreu equation (see (\ref{gAbreu})). In \cite{CHLLS}, Chen-Han-Li-Lian-Sheng  solved the problem for extremal metrics for  homogeneous toric bundles with two-dimensional fibers under the uniform $K$-stability condition. In this paper, we prove the following theorem.

\begin{theorem}\label{theorem_1.1}
Let $(M,\omega)$ be a $n$-dimensional compact toric manifold and $\Delta$ be its
Delzant polytope. Let $G/K$ be a generalized flag manifold and
$G\times_KM$ be the homogeneous toric bundle. Let $\mathbb{D}>0$ be the
Duistermaat-Heckman polynomial, and $A\in C^{\infty}(\bar{\Delta})$ be a given smooth function. Then, $(\Delta, \mathbb{D}, A)$ is relatively $\widehat K$-polystable if and only if there is a smooth
$(G,  T^{n})$-invariant metric $\mathcal{G}$ on $G\times_{K}M$
such that the scalar curvature of $\mathcal{G}$ is $\mathcal{S}=A+h_G$.
\end{theorem}
Refer to the following sections for notations and terminologies. Let $\mathcal{S}$ be a constant or a linear function on $\Delta$. As a consequence of Theorem \ref{theorem_1.1} we have solved the Yau-Tian-Donaldson conjecture of the filtration version for homogeneous toric bundles. Note that even if $\mathcal{S}$ is a constant function, the right hand of equation (\ref{gAbreu}) is not a constant function.

In \cite{Sz2}, Sz\'ekelyhidi gave descriptions of the optimal destabilizer for unstable toric manifolds. We will prove that the convex function of the optimal destabilizer is bounded and gives a filtration.

\begin{remark}
In \cite{Del}, Delcroix expressed $K$-stability of polarized spherical varieties in terms of combinatorial data and provided a combinatorial sufficient condition of $G$-uniform $K$-stability. The appendix of \cite{Del} by Yuji Odaka showed that for non-singular spherical varieties, $G$-uniform $K$-stability is equivalent to existence of cscK metrics. Jubert (see \cite{Jub}) showed the equivalence between the existence of extremal K\"ahler metrics on the total space of semi-simple principal toric fibrations and weighted uniform $K$-stability of the corresponding Delzant polytopes.

The methods in this paper can be applied to explore the problems for which the scalar curvature equations can be reduced to the equations on polytopes, such as smooth spherical varieties and semi-simple principal toric fibrations.
\end{remark}

This paper is organized as follows. In Section \ref{Sec-Stabilities},
we briefly review the notion of relative $\widehat K$-polystability and uniform $K$-polystability. In Section \ref{Sec-Estimates}, we discuss some properties of convex functions on polytopes. In Section \ref{Sec-4}, we prove the Theorem \ref{theorem_1.3}. In Section \ref{Sec-5}, after briefly reviewing homogeneous toric bundles and the generalized Abreu equation, we sketch the proof of Theorem \ref{theorem_1.1}. In Section \ref{Sec_6}, we prove the convex function of the optimal destabilizer is bounded and gives a filtration for unstable toric manifolds.

\section{\bf Preliminaries}\label{Sec-Stabilities}

Let $(M,\omega)$ be a toric manifold of dimension $n$, $\Delta$ be its corresponding Delzant Polytope. Let $\tau: M\to \mathfrak{t}^\ast$ be the moment map, where $\t^\ast$, identified with $(\real^n)^*$, is the  dual of the  $\mathfrak{t}$ which is the Lie algebra of $T^n$. We denote by $M^\circ$ the open dense subset of $M$ defined by $M^\circ = \left\{ p\in M :T^n-\mbox{action is free at } p\right\}.$ We can describe $M^\circ$ in complex (holomorphic) coordinates as
$$M^\circ  = \mathbb R^n \times i T^n = \left\{ x+iy : x\in \mathbb R^n, y\in \mathbb R^n /\mathbb Z^n\right\}.$$
The K\"ahler form is given by a potential $f\in C^\infty(M^\circ)$, the potential $f$ depends only on the $x$ coordinates: $f = f(x) \in C^\infty(\mathbb R^n)$. Since $f$ is a (smooth) strictly convex function on $\mathfrak{t}$, the gradient of $f$ defines a (normal) map $\nabla^f$ from $\mathfrak{t}$ to  $\mathfrak{t}^\ast$:
$$
\xi=(\xi_1,...,\xi_n)=\nabla^f(x) =\left(\frac{\partial f}{\partial x_1},...,\frac{\partial f}{\partial x_n}\right).
$$
The function $u$ on $\mathfrak{t}^\ast$
$$
u(\xi)=\langle x,\xi\rangle - f(x).
$$
is called the Legendre transformation of $f$. We write $u=f^*$. Conversely, $f=u^*$.

In terms of coordinates $\xi$ and the Legendre transform function $u$ of $f$, the scalar curvature can be written as
$$
\mathcal{S}(u)= -\sum U^{ij}w_{ij},
$$
where $(U^{ij})$ is the cofactor matrix of the Hessian matrix $(u_{ij})$, $w=(\det(u_{ij}))^{-1}$.

\vskip 0.1in
Suppose that $\Delta$ is given by
$$
\Delta=\{\xi|l_i(\xi)> 0,\;\;\; 0\leq
i\leq d\}
$$
where $l_i$ are liner functions given by
\begin{equation}\label{eqn_1.1}
l_i(\xi):=\langle\xi,a_i\rangle- \lambda_i.
\end{equation}
Guillemin constructed  a natural K\"{a}hler  form $\omega_g$ and we denote the class by $[\omega_g]$. We call it the Guillemin metric.
Let
\begin{equation}\label{eqn_1.2}
v=g^*=\sum_{i} l_i\log l_i,
\end{equation}
where $l_i$ are linear functions determining the facet of $\Delta$, $g$ is the potential function of the Guillemin metric. The prescribed scalar curvature problem reduces to finding a smooth convex solution in $\Delta$ for the 4-th order PDE
\begin{equation}\label{eqn-1.3}
-\sum U^{ij}w_{ij}=A
\end{equation}
subject to the boundary condition
$$u-\sum_{i} l_i\log l_i \in C^{\infty}(\bar{\Delta}),$$
where $A$ be a smooth function on $\bar\Delta$. \eqref{eqn-1.3} is called the Abreu equation (see \cite{Abr}).
\v
For any smooth function $A$ on $\bar\Delta$, Donaldson defined  a
functional on continuous convex functions on $\bar\Delta$:
$$\mc F_{A}(u)= -\int_{\Delta}\log \det(u_{ij})d\mu + \mc L_{A}(u),$$
where $\mc L_{A}$ is the linear functional
$$\mc L_{A}(u) = \int_{\partial \Delta}u d\sigma - \int_{\Delta}Au
d\mu,$$
where $d\mu$ is the  Lebesgue measure on $\mathbb R^n$, and on each face $F$, $d\sigma$ is a constant multiple of the standard $(n-1)$-dimensional Lebesgue measure (see \cite{D1} for details).

We introduce several classes of functions. Denote $\mathcal P$ the rational piecewise linear convex functions on $\bar\Delta$. Set
\begin{align*}
\mc C&=\{u\in C(\bar\Delta):\, \text{$u$ is convex on $\bar\Delta$ and smooth in $\Delta$}\},\\
\mathbf{S}&=\{u\in C(\bar\Delta):\, \text{$u$ is convex  on $\bar\Delta$ and $u-v$ is smooth on $\bar\Delta$}\},\end{align*}
For a fixed
point $p_o\in \Delta$, we consider
\begin{align*}
\mathcal P_{p_o}&=\{u\in\mathcal P:\, u\geq u(p_o)=0\},\\
{\mc C}_{p_o}&=\{u\in \mc C:\, u\geq u(p_o)=0\},\\
\mathbf{S}_{p_o}&=\{ u\in \mathbf S :\, u\geq u(p_o)=0\}.\end{align*}
We say functions in ${\mc P}_{p_o}$, ${\mc C}_{p_o}$ and ${\mathbf{S}}_{p_o}$ are {\it normalized} at $p_o$. Let
\begin{eqnarray*}
\mc C_\ast&=&\{u| \mbox{there exist a constant $C>0$  and  a  sequence of $\{u_k\}$ in ${\mc P}_{p_o}$ }\\
&&\mbox{such that $\int_{\partial\Delta} u_kd\mu<C$ and
$u_k$  locally uniformly converges to} \\
&& \mbox{$u$ in $\Delta$}\}.
\end{eqnarray*}
Let $P>0$ be a  constant, we define
$$
\mc C_\ast^P=\{u \in\mc C_\ast| \int_{\partial\Delta} u_k d\mu\leq P\}.
$$

\subsection{(Relative) $K$-polystability}
In \cite{D1}, Donaldson introduced the (relative) $K$-polystability by test configurations. For toric case, the definition is

\begin{defn}\label{definition_2.1.1}
 Let $A\in C^{\infty}(\bar\Delta)$ be a smooth function on $\bar\Delta$.
$({\Delta},A)$ is called {\it relatively $K$-polystable} if $\mc
L_{A}(u)\geq 0$ for all rational piecewise-linear convex functions
$u$, and $\mc L_{A}(u)=0$ if and only if $u$ is a linear function.
\end{defn}

\begin{remark}\label{remark 1.3}
Since every continuous convex function on $\bar\Delta$ can be approximated by rational piecewise-linear convex functions uniformly, so if $({\Delta},A)$ is relatively $K$-polystable, we have $\mathcal L_{A}(u)\geq 0$ for every continuous convex function on $\bar\Delta$.
\end{remark}

\subsection{(Relative) $\widehat K$-stability}
An example of Apostolov-Calderbank-Gauduchon-T\o nnesen-Friedman (see \cite{ACGT}) suggests that $K$-stability might not be the correct one for general polarized manifolds. Sz\'{e}kelyhidi introduced  a modified definition of $K$-stability, which is called $\widehat K$-stability (see \cite{Sz1}).
\begin{defn}
Let $(M,L)$ be a polarized variety, $R_k=H^{0}\left(M, L^{k}\right)$, then $R=\bigoplus_{k \geqslant 0}R_k$ is a graded ring. A filtration of $R$ is a  filtration
$$
\mathbb C=F_{0} R \subset F_{1} R \subset \ldots \subset R
$$
such that
\begin{itemize}
\item[1.] $\left(F_{i} R\right)\left(F_{j} R\right) \subset F_{i+j} R$,
\item [2.] $F_iR=\bigoplus_{k \geqslant 0}F_iR\cap R_k$,
\item[3.] $R=\bigcup_{i \geqslant 0} F_{i} R$.
\end{itemize}
\end{defn}

\begin{remark}
Sz\'ekelyhidi observes if the Rees algebra of a filtration $\mathbf F$ is finitely generated, the scheme $Proj_{\mathbb C[t]}Rees(\mathbf F)$ is a test configuration for $(X, L)$. Conversely, a construction of Witt-Nystr\"om defines a filtration from any test configuration. So it is the same to talk about test configurations as filtrations with finitely generated Rees algebra (see \cite{Sz1}, \cite{WN}, \cite{BHJ}, \cite{D6}).
\end{remark}

Let $u$ be a convex function representing a filtration $\chi_{u}$. Following \cite{Sz1},
let $u_k$ be the largest convex functions which on the points $\alpha\in\Delta\cap\frac{1}{k}\mathbb Z^n$ is defined by
$$u_k(\alpha)=\frac{1}{k}\lceil ku(\alpha)\rceil,\;(k=1,2,...).$$
Then the test-configurations obtained by the rational piecewise-linear functions $u_k$ is an approximation to the filtration $\chi_{u}$.


We consider the toric cases. For any convex function $u:\bar\Delta\to \mathbb R$, we can define a filtration $\chi_u$, (see \cite{Sz1}). This filtration arises from a test-configuration if and only if $u$ is rational piecewise linear.

By the definition of norms of filtrations (see \cite{Sz1}), we have
$$\|\chi_u\|^2=\inf_{\ell}\left\{\int_{\Delta}(u+\ell)^{2}d\mu-\frac{\left(\int_{\Delta}(u+\ell)d\mu\right)^2}{Vol(\Delta)}\right\},$$
where $\ell$ runs through all the affine functions.
Then we can define the following $\widehat K$-stability for the toric case.

\begin{defn}\label{defn_toric_hat_k}
 Let $A\in C^{\infty}(\bar\Delta)$ be a smooth function on $\bar\Delta$. $(\Delta,A)$ is called {\it relatively $\widehat{K}$-polystable} if $\mc L_{A}(u)\geq 0$ for all convex  functions $u:\bar\Delta\to \mathbb R$, and $\mc L_{A}(u)=0$ if and only if $u$ is a linear function in $\Delta$.
\end{defn}


\begin{remark}
It is known that every continuous convex function $u$ defined in $\Delta$ has a lower semi-continuous regularization, and the lower semi-continuous regularization of $u$ is the smallest convex function which equals to $u$ in $\Delta$. Thus, in the definition of relative $\widehat{K}$-polystability, we can let $u$ be lower semi-continuous convex functions on $\bar\Delta$.
\end{remark}

When $A=a$ is a constant function, Definition \ref{defn_toric_hat_k} is the same as the definition for the toric case in \cite{Sz1}.

\subsection{Uniform $K$-polystability}
In \cite{D2}, Donaldson also proposed a stronger version of stability which we call   {\em uniform stability}.
\begin{defn}\label{definition_2.1.5}
$({\Delta},A)$ is called {\it uniformly relatively $K$-polystable}
if for any $u\in {\mc P}_{p_o}(\Delta)$
$$
\mc L_A(u)\geq \lambda\int_{\partial \Delta} u d \sigma
$$
for some constant $\lambda>0$. Sometimes, we say that $\Delta$ is $(A,\lambda)$-stable.
\end{defn}


The following lemma is Lemma 3.1 in \cite{CLS4}.
\begin{lemma}\label{lemma_1.2.12}
For any lower semi-continuous function $u$ in $\mc C_\ast^P$, there is a sequence of functions $u_k\in \mc C$
such that  $u_k$ locally uniformly converges to $u$ in $\Delta$ and
\begin{eqnarray}
&&\int_{\partial\Delta} u d\sigma=\lim_{k\to\infty}\int_{\partial\Delta} u_kd\sigma, \label{eqn_A}\\
&&\mc L_A(u)=\lim_{k\to\infty}\mc L_A(u_k).\label{eqn_B}
\end{eqnarray}
\end{lemma}

\section{\bf Estimates of convex functions}\label{Sec-Estimates}

Let $o$ be the center of $\Delta$. Choose coordinate system $\xi_i$ such that $\xi(o)=0.$ Set
$$\mathcal Q=\left\{u\in\mathcal C_{*}^{P}| u\text{ is lower semi-continuous}, \int_{\partial \Delta }u d\sigma =1, u(o)=\inf_{\bar{\Delta}}u=0.\right\}.$$

Let $D_{2a_{o}}(0)$ be a ball in $\Delta.$ By Lemma 5.2.3 in \cite{D1}, there exists a constant $C_1>0$ such that for any $u\in\mathcal Q$,
\begin{equation}
|u|_{L^{\infty}(D_{a_{o}}(0))}\leq C_1.
\end{equation}

Let $f$ be the Legendre transformation of $u$, that is,
\begin{equation}\label{leg_tran}
f(x)=\sup _{\xi\in\bar{\Delta }}(\langle x,\xi\rangle-u (\xi)),\;\;x\in\mathbb R^n.
\end{equation}
It is well known that $f$ is a lower semi-continuous convex function in $\mathbb R^n$. Denote by $u^*=f.$
\begin{remark}
We can extend $u$ to be defined on $(\mathbb R^n)^*$ by defining
$$u|_{(\mathbb R^n)^*\setminus \bar{\Delta}}=+\infty.$$
Then the Legendre transformation of $u$ can be defined by
\begin{equation}\label{leg_tran-infty}
f(x)=\sup _{\xi\in(\mathbb R^n)^* }(\langle x,\xi\rangle-u(\xi)).
\end{equation}
It is easy to see that $f$ in \eqref{leg_tran} is the same as in \eqref{leg_tran-infty}.
\end{remark}
Since $u$ is convex, lower semi-continuous and $u(o)<\infty$, by Theorem 6.1.2 in \cite{N-P},  $u$ is the Legendre transformation of $f$, that is, $u^{**}=u$ and
\begin{equation}\label{leg_tran-f}
u(\xi)=\sup _{x\in \mathbb R^n } (\langle\xi, x\rangle-f (x) ) .
\end{equation}

Denotes by $Df$ the normal map (or sub-differential) of the convex function $f$, that is, for $p\in\mathbb R^n$,
 $$
Df(p):=\{q\in(\mathbb R^{n})^*\;| \;f(x)\geq \langle q,(x-p)\rangle+f(p)\text { for all } x \in \mathbb R^n\}.
$$
Similarly, we denotes by $D u$ the normal map of $u$, i.e., 
 $$
Du(q):=\{p\in\mathbb R^{n}\;| \;u(\xi)\geq\langle p,(\xi-q)\rangle+u(q)\text { for all } \xi \in (\mathbb R^{n})^*\}.
$$

Since $u(\xi)\geq 0$ and $\bar\Delta$ is bounded, for all $x\in\mathbb R^n$, we have $ f(x)=\sup _{\xi\in \bar\Delta}(\langle x,\xi\rangle  -u (\xi))<+\infty$. So $Df(p)\neq \emptyset,\forall p\in \mathbb R^n.$
For any $p\in \mathbb{R}^n$ and $q\in Df(p)$, we denote
\begin{equation}
\label{normal_sect}
S_{f}(p,q,\epsilon):=\{x\in \mathbb{R}^n\;|\;f(x)\leq f(p)+\langle q, (x-p)\rangle+\epsilon\}.
 \end{equation}

We need the following lemma, see \cite{N-P} Proposition 6.1.1 for a proof.
\begin{lemma} \label{normal_m}
For any $p\in \mathbb R^n,$ $q\in Df(p)$ if and only if
\begin{equation}\label{eqn-3.6}
u(q)= \langle q, p\rangle -f (p).
\end{equation}
Similarly, $p\in Du(q)$ is also equivalent to \eqref{eqn-3.6}.
\end{lemma}

By  \eqref{leg_tran} and $u\geq 0$, we have $f(x)\geq 0-u(0)=0$ and $f(0)=\sup_{\xi\in \bar\Delta}(-u(\xi))=0.$ Obviously, $f(0)+u(0)=0$ and $0\in D f(0)$.

\v
We denote $S_{f}(0,0,h)$ by $S_{f}(0,h)$.
  Note that for any $e\in S^{n-1}(1)$ and $t>0$,
$$
f(te) \geq te\cdot a_oe -u(a_oe) \geq a_ot-C_1.
$$
So we have
\begin{lemma}
For any $h>0,$ $\overline{S_{f}(0,h)}$ is compact.
\end{lemma}

For $ Df(\mathbb{R}^n)$, we can prove that
\begin{lemma} \label{im_normal}
$\Delta\subset Df(\mathbb{R}^n)\subset\bar\Delta$.
\end{lemma}
\begin{proof}
Since $u(\xi)=+\infty$, for all $\xi\in(\mathbb R^n)^*\setminus\bar\Delta$, by Lemma \ref{normal_m}, we have $Df(\mathbb{R}^n)\subset\bar\Delta$, so we only need to prove that $\Delta\subset Df(\mathbb{R}^n)$. If it is not true,  there exists a point $q\in \Delta\setminus Df(\mathbb{R}^n) $. Since $u(q)<+\infty$, there exists  $p\in Du(q)$, it follows from Lemma \ref{normal_m} that $q\in D f (p)$. We get a contradiction.
\end{proof}

\begin{theorem}\label{theorem_stab}
Assume that $(\Delta,A)$ is relatively $\widehat K$-polystable. Let $u\in \mathcal Q$ and  $\mathcal L_{A}(u)=0$.  Then
\begin{equation}
u\in L^{\infty}(\bar{\Delta}).
\end{equation}
\end{theorem}
\emph{Proof.}  We discuss two cases.

{\bf Case 1.} There exists a constant $h_{o}>0$ such that for any $h\geq h_{o}$, $D f(\overline{S_{f}(0,h)})=\bar{\Delta}. $

For any $q\in \bar \Delta$, there exists $p\in \overline{S_{f}(0,h)}$ such that $q\in Df(p)$.  By Lemma \ref{normal_m} we have
$$
u(q)=\langle q, p\rangle-f(p)\leq diam(\Delta) diam(\overline{S_{f}(0,h_o)}).$$
Then  we have
$$
|u|_{L^\infty(\bar{\Delta})}<+\infty.
$$
\v\n

{\bf Case 2.} For any $h>0$, we have
\begin{equation}\label{eqn_ff}
 \bar \Delta\setminus Df(\overline{S_{f}(0,h)})\neq \emptyset.
\end{equation}
Set  $W_{h}=Df(\overline{S_{f}(0,h)})$ and $B_{h}=Df(\partial\overline{S_{f}(0,h)}).$ It is well known that $Df(\overline{S_{f}(0,h)})$ is compact (see Lemma 1.1.3 in \cite{Gui}). Obviously,
$$B_h \subset W_{h},\;\;\;\;\bar\Delta\setminus  W_{h}\neq \emptyset,\;\;\;\forall h>0.$$

We need some characters of $B_h$ and $W_h$.
Note that  $B_h$ may not be the boundary of $W_h$. We have
\begin{lemma}
$\p W_h\subset B_h.$
\end{lemma}
\begin{proof}
 Assume that the lemma is not true. Then there exists $q\in \p W_h$ such that $D u(q)\cap \partial\overline{S_{f}(0,h)}=\emptyset$.
Obviously, $D u(q) \subset\subset   S_{f}(0,h)$.  It is well known that $D u(q)$ is a closed convex set. Take a point $p\in D u(q).$  By Lemma \ref{normal_m}, $q\in Df(p).$ If $S_{f}(p,q,0) \cap \p S_{f}(0,h)\neq \emptyset,$ let $p_1\in S_{f}(p,q,0) \cap \p S_{f}(0,h),$ we have
$$
f(p_1)=\langle q,(p_1-p)\rangle +f(p).
$$
Combining this and  the definition of normal map, we have
$$
f(x)\geq \langle q, (x-p)\rangle+f(p)=\langle q,(x-p_1)\rangle+f(p_1).
$$
Then $q\in B_h.$ We get a contradiction. Suppose that  $S_{f}(p,q,0) \cap \p S_{f}(0,h)=\emptyset$, since $S_{f}(p,q,0)$ is a closed set, $\p S_{f}(0,h)$ is a compact set, there exists a constant $\epsilon>0$ such that
$$S_{f}(p,q,\epsilon)\subset\subset S_{f}(0,h).$$
By the convexity we have $$D_\delta(q)\subset D f(S_{f}(p,q,\epsilon)) \subset D f(S_{f}(0,h)),$$
for some constant $\delta>0.$
In particular, $q$ is an interior point of $W_h.$ We get a contradiction. The lemma is proved.
\end{proof}

By Lemma \ref{normal_m}, we have
$$
\partial\overline{S_{f}(0,h)} \subset Du(B_h).
$$

\begin{lemma}\label{int_oq}
 Let $oq$ be a line segment with $q\in \partial \Delta$. Then $oq\cap W_h$ is  connected and $oq\cap (\bar\Delta\setminus W_h)$ is connected.
\end{lemma}
\begin{proof} First, we prove that $oq\cap B_h$ is connected.
Let $(\xi,\xi_{n+1})$ be the coordinates of $\bar\Delta\times \mathbb R.$ Denote by $\Gamma_{u}$ the graph of the function $u$ in $\bar\Delta\times \mathbb R.$
 For any $q_1,q_2\in oq\cap B_h,$ set $Q_i=(q_i,u(q_i)),i=1,2.$ Denote $P=(0,-h).$
Let $p_1,p_2\in \p S_{f}(0,h)$ be the points with $q_{i}\in D f(p_i),i=1,2.$ Consider the supporting hyperplane $l_{q_{i}}$ defined by the function
$$
l_{q_{i}}: \quad \xi_{n+1}=\langle p_i, (\xi-q_i)\rangle+u(q_i),\;\;\;i=1,2.
$$
Since $\langle p_i, q_i\rangle=u(q_i)+f(p_i)$ and $f(p_i)=h$, we have $P\in l_{q_{i}}.$ Obviously, $Q_i\in l_{q_i}.$
Since the line $PQ_i$ is tangent to  $\Gamma_{u}$ at $Q_i$, by the convexity of $u$, we have $P,Q_1,Q_2$ are in the same line, $u|_{q_1q_2}$ is affine, and $l_{q_{1}}=l_{q_{2}}$. So $\langle p_{1},\xi\rangle|_{oq}=\langle p_2,\xi\rangle|_{oq}$. If we denote by $l_{q}$ the hyperplane which is defined by $\xi_{n+1}=\langle p_1,(\xi-q)\rangle+u(q)$, $l_{q}$ is a supporting hyperplane of the graph of $u$ at every point between $q_1$ and $q_2$. Thus, $p_1\in Du(tq_1+(1-t)q_2)$, for all $0\leq t\leq 1$. So the segment $q_1q_2\subset B_{h}$. We have proved $oq\cap B_h$ is connected. Next, we prove that $oq\cap W_h$ is a line segment. In fact, if this is not true, there exists a point $q_0$ such that $q_0\notin W_h$, $W_h\cap oq_0\neq\emptyset$ and $W_h\cap q_0q\neq\emptyset$. Then we have $B_h\cap oq_0\neq\emptyset$, $B_h\cap q_0q\neq\emptyset$. By the connectedness of $B_h\cap oq$, we have $q_0\in B_h$, which is a contradiction. so $oq\cap W_h$ is connected, $oq\cap (\bar\Delta\setminus W_h)$ is connected.
\end{proof}

\begin{remark}
We can describe $\bar\Delta$ and $\bar\Delta\setminus W_h$ in the polar coordinates. Let $F_i,i=1,\cdots,d$ be the facets of $\bar{\Delta}$. We can divide $\Delta$  into several cones with vertex $o$ and bases $F_i$. Denote them by $D_{F_1},\cdots,D_{F_d},$ The cones also divide $S^{n-1}(1)$ into $S_{F_i},i=1,\cdots,d.$ Then $D_{F_i}$ can been written as
$$
\{(r,e) : 0\leq r\leq R_e, e \in \overline{S_{F_i}}\}.
$$
Let $n_{F_i}$ be the prime normal vector of the facet $F_i.$ Then it is easy to see that
\begin{equation}\cos(e, n_{F_i})\geq c_o, \forall i=1,...,d
\end{equation}
for any $e\in S_{F_i}$, where $c_{o}=\frac{d(o,\partial\Delta)}{diam(\Delta)}$. For any $e\in S^{n-1}(1),$ let $q_{e}$ be the point on $\p \Delta$ with the direction $e.$ By Lemma \ref{int_oq} we know that $oq_{e}\cap W_{h}$ is a line segment. Let $r_{e}$ be the length of $oq_{e}\cap W_{h}$, $R_{e}$ be the length of $oq_{e}$. by Lemma \ref{im_normal} we have
\begin{equation}\label{rad_est}
\lim_{h\to\infty}r_{e}=R_{e}.
\end{equation}
\end{remark}
 \v\n

For any $h>0,$ we define  a functions
$$
u_h(\xi) =\sup_{q\in W_h} l_{q}(\xi),
$$
where $l_{q}$ denotes  the affine-linear function defining a
supporting hyperplane of the graph of $u$ at $q$.
Then $u_h$ is a convex function and $u_h|_{W_h}=u|_{W_h}.$

\begin{lemma}\label{est_u_h}
 Let $oq$ be a line segment with $q\in \partial \Delta$. Suppose that  $oq\cap W_h=oq_1$. Then $u$ is linear on $q_1q$, and for any $\xi\in q_1 q$,
\begin{equation}
u_h(\xi)=u_h(q_1)+\frac{|q_1\xi|}{|oq_1|}(u(q_1)+h).
\end{equation}
\end{lemma}
\begin{proof}
Since $q_1\in\partial W_h$, we have there exists $p_1\in\partial\overline{S_{f}(0,h)}$ such that $l_{q_{1}}:\xi_{n+1}=\langle p_1,(\xi-q_1)\rangle+u(q_1)$ is a hyperplane of the graph of $u$ at $q_1$, and $(0,-h)\in l_{q_{1}}$. When restrict $oq$, the equation of $l_{q_1}$ is $\xi_{n+1}=u_h(q_1)+\frac{|q_1\xi|}{|oq_1|}(u(q_1)+h)$, so $u_h(\xi)\geq u_h(q_1)+\frac{|q_1\xi|}{|oq_1|}(u(q_1)+h)$ for any $\xi\in q_1 q$.

On the other hand, for any $\xi'\in W_h$ and any supporting hyperplane $l_{\xi'}$ of the graph of $u$ at $\xi'$, there exists $x\in\overline{S_{f}(0,h)}$ such that the function of $l_{\xi'}$ is $\xi_{n+1}=\langle x, (\xi-\xi')\rangle+u(\xi')$. Since $u(\xi')+f(x)=\langle x,\xi'\rangle$ and $f(x)\leq h$, it follows that  $-h\leq l_{\xi'}(0)\leq 0.$ Since  $l_{\xi'}(q_{1})\leq u(q_1)=l_{q_{1}}(q_1)$, we have $l_{\xi'}(\xi)\leq l_{q_{1}}(\xi)$ for any $\xi\in q_1 q$. So $u_h(\xi)\leq u_h(q_1)+\frac{|q_1\xi|}{|oq_1|}(u(q_1)+h)$ for any $\xi\in q_1 q$. The lemma is proved.
\end{proof}

\begin{lemma}\label{cont_u}
Assume that  $\bar\Delta\setminus D f(\overline{S_{f}(0,h)})\neq \emptyset$ for any $h>0.$ Then
$$\mathcal L_{A}(u) >\mathcal L_{A}(u_h),$$
as $h$ large enough.
\end{lemma}
\begin{proof}
Let $oq$ be a ray with $q\in \partial \Delta$.  Note that $u\geq u_h$. Then we have $(u-u_h)|_{oq}$ is a non-negative lower semi-continuous convex function. Denote $v_h=u-u_h$, then
$$\mathcal L_{A}(u) -\mathcal L_{A}(u_h)=\int_{\p \Delta \cap (\bar{\Delta}\setminus W_h)} v_h d\sigma -\int_{\Delta\setminus W_h} Av_h d\mu. $$
We use the polar coordinates to calculate $\int_{\Delta\setminus W_h} v_h d\mu.$ At first, we calculate $\int_{D_{F_i}\cap (\Delta\setminus W_h)} v_h d\mu$ for any facet $F_i$. Let $d\Theta$ be the standard volume form of the sphere $S^{n-1}(1).$ Then
\begin{align} \nonumber
\int_{D_{F_i}\cap (\Delta\setminus W_h)} v_h d\mu&=\int_{S_{F_i}}\int_{r_{e}}^{R_{e}}r^{n-1}v_hdr d\Theta \\
&\leq \int_{S_{F_i}}R_{e}^{n-1}v_h(R_{e},e) d\Theta \nonumber \\
&\leq \max_{e\in S^{n-1}(1)} (R_{e}-r_{e})\int_{S_{F_i}}R_{e}^{n-1}v_h(R_{e},e) d\Theta. \label{est_int}
\end{align}
On the other hand, 
\begin{align}
\int_{F_i\cap (\bar\Delta\setminus W_h)} v_h d\sigma&=\frac{1}{|n_{F_i}|}\int_{S_{F_i}} v_h(R_{e},e)R^{n-1}_{e}\cos(n_{F_i},e)d\Theta\nonumber\\
&\geq \frac{c_o}{|n_{F_i}|}\int_{S_{F_i}} v_h(R_{e},e) R^{n-1}_{e}d\Theta.\label{est_bou}
\end{align}
By \eqref{rad_est} we can choose $h$ large enough such that
$$
\max_{e\in S^{n-1}(1)} (R_{e}-r_{e} ) \leq \frac{c_o}{2\max_{\bar\Delta}|A|\max_{i=1,...,d}|n_{F_i}|}.
$$
Then,
\begin{align}
\mathcal L_{A}(u) -\mathcal L_{A}(u_h)&\geq\sum_{i}\left(\left(\frac{c_o}{|n_{F_i}|}-\max_{\bar\Delta}|A|\max_{e\in S^{n-1}(1)} (R_{e}-r_{e})\right)\int_{S_{F_i}}R_{e}^{n-1}v_h(R_{e},e) d\Theta\right)\nonumber\\
&\geq\frac{c_o}{2\max_{i=1,...,d}|n_{F_i}|}\int_{S^{n-1}(1)}R_{e}^{n-1}v_h(R_{e},e) d\Theta.\nonumber
\end{align}
Since $v_h\geq 0$ and $v_h$ is lower semi-continuous, we have that $\int_{S^{n-1}(1)} v_h(R_{e},e) R^{n-1}_{e}d\Theta >0$. In fact, if $\int_{S^{n-1}(1)} v_h(R_{e},e) R^{n-1}_{e}d\Theta=0$, we have $\int_{\Delta} v_hd\mu=0$. So $v_h=0$ almost everywhere in $\Delta$, then $v_h\equiv 0$ in $\bar\Delta$. Thus $\bar\Delta\subset D f(\overline{S_{f}(0,h)})$, which is a contradiction.

Thus
$$
\int_{\p\Delta} v_hd\sigma -\int_{\Delta} v_h d\mu >0.
$$
The lemma is proved.
\end{proof}

{\bf Continuing the proof of Theorem \ref{theorem_stab}.}  It follows from $\mathcal L_{A}(u)=0$ and Lemma \ref{cont_u} that $\mathcal L_{A}(u_h)<0.$ This contradicts to the assumption of relative $\widehat K$-stability. So Case 2 can not happen. The theorem follows.

\section{\bf Proof of Theorem \ref{theorem_1.3}}\label{Sec-4}

In this section we prove Theorem \ref{theorem_1.3}. Firstly, we prove the following     result.

\begin{theorem}\label{theorem_5.2}
Let $(M,\omega)$ be a $n$-dimensional compact toric manifold and $\Delta$ be its Delzant polytope. Then $(\Delta, A)$ is relatively $\widehat{K}$-polystable if and only if there is $\lambda>0$ such that
$(\Delta,A)$ is {\it uniformly relatively $K$-polystable}.
\end{theorem}
\begin{proof}
$\Rightarrow$: If $(\Delta,A)$ is not uniformly relatively $K$-polystable, then for any integer $k > 0$ there exists a function $u_k\in {\mc P}_{p_o}(\Delta)$ with $\int_{\partial\Delta}u_{k}d\sigma=1$ such that
$$\mc L_A(u_{k})\leq\frac{1}{k}.$$
Then there exists a subsequence (still denoted by $u_k$) locally uniformly converges to a lower semi-continuous convex function $u$. By proposition 5.2.6 in \cite{D1}, we have $\int_{\partial\Delta}ud\sigma\leq 1$.

{\bf Claim:} $\int_{\partial\Delta}ud\sigma=1$.

In fact, if $ \int_{\p\Delta }ud\sigma < 1$, since $\int_{\Delta }Aud\mu=\lim_{k\to\infty}\int_{\Delta }Au_{k} d\mu=1$, we have
\begin{equation}\label{eqn_4.16}
\mathcal{L}_{A}(u) <0.
\end{equation}
By Lemma \ref{lemma_1.2.12}, we can find a new sequence $\{\check u_{k}\}\subset\mathcal C$, which locally uniformly converges to $u$, such that
$$
\lim_{k\to \infty} \mathcal{L}_{A}(\check u_{k})=\mathcal{L}_{A}(u) <0.
$$
We get a contradiction with relative $\widehat{K}$-polystability as $k$ large enough. So we proved the claim.

By Theorem \ref{theorem_stab}, $u$ is bounded, so $u$ represents a filtration. Since $(\Delta,A)$ is relatively $\widehat{K}$-polystable, this implies that $u\in\mc G$ in $\Delta$. Since $u_k$ are normalized, $u\equiv 0$ in $\Delta$. However, by the assumption, we know that
$$\int_{\Delta}Au_kd\mu\geq 1-\frac{1}{k}$$
and the limit of this is $\int_{\Delta}Aud\mu\geq 1$, which is impossible. This completes the proof of one direction.

$\Leftarrow$: Let $u$ be a normalized convex function representing a filtration, By the construction in \cite{Sz1}, there exists a sequence $u_{k}\in\mc P$ such that $\{u_k\}$ decrease, uniformly converges to $u$ and $\mc L_A(u_k)\to\mc L_A(u)$. Hence $\mc L_A(u)\geq 0$.

Now suppose that $u$ is a nontrivial extremal function, i.e, $u\not\equiv 0$ in $\Delta$ and $\mc L_A(u)=0$. Then $\lim_{k\to\infty}\mathcal L_{A}(u_k)=0$. If $\int_{\partial\Delta} u d\sigma=0$, since $u$ is normalized, we have $u\equiv 0$ in $\Delta$. So we can assume that
$$\int_{\partial\Delta} u d\sigma=\lim_{k\to\infty}\int_{\partial\Delta} u_k d\sigma=C>0.$$
So
$$\mc L_A(u_k)\geq\lambda  \int_{\partial\Delta} u_k d\sigma\to\lambda C>0.$$
We get a contradiction. Hence $u$ can not be a nontrivial extremal function. This completes the proof.
\end{proof}

\begin{remark}
If $A$ is a constant function, by  Theorem \ref{theorem_5.2} and the existence of cscK metrics under uniform stability condition for toric manifolds (see \cite{CC2}, \cite{Li}, \cite{His}), we get the sufficiency of relative $\widehat K$-polystability for the cscK case.
\end{remark}

By Theorem \ref{theorem_5.2}, we need to prove the necessity and sufficiency of uniformly relative $K$-polystability. By Theorem 4.6 in \cite{CLS4} and Lemma \ref{lemma_1.2.12},  uniform relative $K$-polystability is a necessary condition (also see \cite{His}). In the following, we prove the sufficiency of uniform relative $K$-polystability.

\subsection{\bf The Mabuchi functional and $d_1$ metrics}
\v
In this subsection, we would use the same letter $C$ to denote different constants depending only the manifold $M$ and the background metric.

Let $(M,g)$ be a K\"ahler manifold, the Mabuchi functional $\M$ is a real-valued
function on the set of K\"ahler metrics in the same K\"ahler class $[\omega_{g}]$. The functional is defined through the formula for its variation at a metric $\omega=\omega_{g}+\sqrt{-1}\partial\bar{\partial} \phi$:
$$
\delta\mathcal{M}=\int_{M}(\mathcal{S}(\omega)-\underline{\mathcal{S}})\delta \phi \frac{\omega^{n}}{n!},
$$
where $\underline{\mathcal{S}}=\frac{\int_{M}S(\omega) \frac{\omega^{n}}{n!}}{\int_{M}  \frac{\omega^{n}}{n!}}
$ is the average value of the scalar curvature.
Denote by $Ric(\omega_{g})$ to be the Ricci form of $\omega_{g}$. In \cite{Che} Chen proved that the Mabuchi functional can be written as
$$
\M(\omega) =\int_{M} \log \left(\frac{\omega^{n}}{\omega^{n}_{g}}\right)  \frac{\omega^n}{n!} +J_{-Ric(\omega_{g})}(\omega)   .
$$
Here $\left.\int_{M} \log \left(\frac{\omega^{n}}{\omega^{n}_{g}}\right)  \frac{\omega^n}{n!} \right.$
is the entropy, and $J_{-Ric(\omega_{g})}(\omega)$ is defined by
$$
J_{-Ric(\omega_{g})}(\omega)=-\frac{1}{n!}\int_{M}\phi \sum_{k=0}^{n-1}Ric(\omega_{g})\wedge \omega_{g}^{k} \wedge \omega^{n-1-k} +\underline{\mathcal{S}}\frac{1}{(n+1)!} \int_{M} \phi \sum_{k=0}^{n} \omega_{g}^{k} \wedge \omega^{n-k}  . $$
\v
For any $p\geq 1,$ Guedj-Zeriahi (in \cite{GZ}) introduced the following space:
 $$
 \E^{p}(M,\omega_{g})=\left\{ \phi\in PSH(M,\omega_{g})\;|\;  \int_{M}\omega_{\phi}^{n}=\int_{M}\omega_{g}^{n},\int_{M}|\phi|^{p} \omega_{\phi}^{n} <\infty \right\}.
 $$
Darvas (\cite{Dar}) introduced the metric on $\H$ by
$$
\|\delta \phi\|_{\phi}=\int_{M}|\delta \phi| \frac{\omega^{n}_{\phi}}{n!},\;\;\;\;\;\forall \;\delta \phi\in T_{\phi}\H
$$
which induced the path-length distance $d_{1}$ on the space $\H.$
He proved that $(\H,d_{1})$ is a metric space,
$ (\E^{1}(M, \omega_{g}), d_{1})$ is a geodesic metric space, which is the metric completion of $(\H,d_{1})$. Darvas also obtained the following estimate in (\cite{Dar})
\begin{equation}\label{Dar-est}
\int_{M} |\phi_{1}-\phi_{2}|\frac{\omega_{1}^{n}}{n!} +\int_{M} |\phi_{1}-\phi_{2}|\frac{\omega_{2}^{n}}{n!}\leq C d_{1}(\phi_{1},\phi_{2})  ,\;\;\; \mbox{ for  any }  \phi_{1},\phi_{2}\in \E^{1}(M,\omega_{g}),
\end{equation}
for some  constant $C > 0$ depending only on $n$, where $\omega_{i}=\omega_{g}+\sqrt{-1}\p\bar{\p} \phi_{i},i=1,2$.
\v

Now we consider the toric case. The metric $\omega_{g}$ is $T^{n}$-invariant. We denote by $\H_{T^{n}}$ the subspace of $T^{n}$-invariant functions in $\H.$
Set $a=\frac{Vol(\partial \Delta,d\sigma ) }{Vol( \Delta,d\mu ) }$. Donaldson have proved that the relation of $\M(\omega_{f})$ and $\F_{a}(u)$ (see \cite{D1} Proposition 3.2.8.):
\begin{equation}\label{eqn_Ma_Do}
\M(\omega_{f})=(2\pi)^{n} \F_{a}(u).
\end{equation}
Suppose that $\phi_{1},\phi_{2}\in \H_{T^{n}}$ are determined by the convex functions $f_{1},f_{2}$ with Legendre transforms $u_{1},u_{2}$. Guedj proved that  (\cite{Gue}  Proposition 4.3) that
\begin{equation}\label{Gue-tor}
d_{1}(\phi_{1},\phi_{2})=\int_{\Delta} |u_{1}-u_{2}| d\mu.
\end{equation}
By the estimates \eqref{Dar-est} and \eqref{Gue-tor}, we have
\begin{equation}\label{est-H}
\int_{M} |\phi_{1}-\phi_{2}|\frac{\omega_{1}^{n}}{n!} +\int_{M} |\phi_{1}-\phi_{2}|\frac{\omega_{2}^{n}}{n!}\leq C\int_{\Delta} |u_{1}-u_{2}| d\mu,\;\;\;\;
\end{equation}
for any $\phi_{1},\phi_{2}\in \H_{T^{n}}.$
\v
Assume $u_{A}$ is a solution of (\ref{eqn-1.3}), Donaldson (see \cite{D1}) proved that $\F_{A}(u_{A})=\inf_{u\in\mc C}\F_{A}(u).$
By adding linear functions we can assume that $u_{A}(o)=0, u_{A}\geq 0$ (resp. $v(o)=0, v\geq 0$) where $o$ is the center of $\Delta.$
 Then we have
$$
\int_{\Delta} |u_{A}|d\mu= \int_{\Delta} u_{A}d\mu\leq C.
$$
Hence
\begin{equation}\label{eqn_bd_u}
\int_{\Delta} |u_{A}-v|d\mu\leq  \int_{\Delta} u_{A}d\mu+\int_{\Delta} vd\mu\leq C.
\end{equation}
Note that
$$
|\F_{A}(u_{A})-\F_{a}(u_{A})|\leq \int_{\Delta} |A-a|u_{A} d\mu\leq C\max_{\bar\Delta}|A-a|.
$$
Then $\F_{a}(u_{A})$ is bounded. Let $f$ be the Legendre transform of $u_{A}$.  By \eqref{eqn_Ma_Do}, we conclude that $\M(\omega_{f})$ is bounded.
\v
Next, we estimate $\left\|J_{-Ric(\omega_{g})}(\omega_{f})\right\|$. By the same calculation as in \cite{CC2},
\begin{align*}
&\left|\int_{M} \phi \sum_{k=0}^{n-1}Ric(\omega_{g})\wedge \omega_{g}^{k}\wedge \omega^{n-1-k}_{f}-n\int_{M}\phi Ric(\omega_{g})\wedge \omega_{g}^{n-1}\right| \\
=&\left|\sum_{s=0}^{n-2}\int_{M} -\p\phi \wedge  \sqrt{-1} \bar{\p}\phi \wedge Ric(\omega_{g}) \wedge  (n-1-s) \omega^{n-2-s}_{g}\wedge \omega^{s}_{f}\right| \\
\leq & n\|Ric(\omega_{g})\|_{\omega_{g}}\left(\int_{M} |\phi|\omega^{n}_{g}+\int_{M} |\phi|\omega^{n}_{f}\right).
\end{align*}
Similar we have
$$
\left| \int_{M} \phi \sum_{k=0}^{n} \omega_{g}^{k} \wedge \omega_{f}^{n-k}  -  (n+1)\int_{M} \phi  \omega_{g}^{n}   \right| \leq C \left(\int_{M} |\phi|\omega^{n}_{g}+   \int_{M} |\phi|\omega^{n}_{f}\right),
$$
for some constant $C$ depending only on $n$. Note that
$$
\left| \int_{M}\phi Ric(\omega_{g})\wedge \omega_{g}^{n-1}\right| +\left|\underline{\mathcal{S}}\int_{M} \phi  \omega_{g}^{n}\right| \leq  \left(\|Ric(\omega_{g})\|_{\omega_{g}}+\left|\underline{\mathcal{S}} \right|\right)  \int_{M} |\phi|\omega^{n}_{g}.
$$
Then we have
\begin{align*}
\left\|J_{-Ric(\omega_{g})}(\omega_{f})\right\|\leq  &C(n)\left(\|Ric(\omega_{g})\|_{\omega_{g}}+\left|\underline{\mathcal{S}}\right|\right) \left(\int_{M} |\phi|\omega^{n}_{g}+   \int_{M} |\phi|\omega^{n}_{f}\right) \\
 \leq & C(n)\left(\|Ric(\omega_{g})\|_{\omega_{g}}+\left|\underline{\mathcal{S}}\right|\right)  \int_{\Delta} |u_{A}-v|d\mu  \leq C,
\end{align*}
where $C(n)$ is a constant depending only on $n$. Here we used \eqref{est-H} in the last inequality. Then by the bound of $\M(\omega_{f})$  we have
$$
\int_{M} Fe^{F} \omega^{n}_{g}=\int_{M} \log \frac{\omega_{f}^{n}}{\omega_{g}^{n}}\omega^{n}_{f} \leq C,
$$
where $F=\log \frac{\omega_{f}^{n}}{\omega_{g}^{n}}$. Thus we get the following theorem,
\begin{theorem}\label{theorem 5.1}
If $\Delta$ is $(A,\lambda)$-stable, then the entropy $\int_{M}Fe^F\omega_{g}^{n}$ is bounded.
\end{theorem}

\def \x {x_{1},\cdots,x_n}
\def \xx  {\xi_{1},\cdots,\xi_n}
\v
We will use the continuity method to prove Theorem \ref{theorem_1.3}. Let $A$ be the scalar function on $\bar\Delta$.
Let $I=[0,1]$ be the unit interval. At $t=0$ we start with Guillemin
metric. Let $A_0$ be its scalar curvature
on $\Delta$. Then $\Delta$ must be $(A_0, \lambda_0)$ stable for some constant
$\lambda_0>0$.  On $\Delta,$ set
$$
A_t=tA+(1-t)A_0, \;\;\; \lambda_t=t\lambda+(1-t)\lambda_0.
$$
It is easy to verify that $\Delta$ is $(A_t,\lambda_t)$-stable for any $t\in[0,1]$.

\v
Set $\Lambda=\{t|\mathcal{S}(u)=A_{t} \mbox{ has a solution in }\mathbf{S}\}$, we should show $\Lambda$ is open and closed.
Openness is standard by using LeBrun and Simanca's argument (\cite{LeSi}). For closeness, by Theorem \ref{theorem 5.1}, the entropy is bounded. By Theorem 1.2 in \cite{CC1}, all derivatives of $\varphi$ is bounded, By Proposition 1.1 in \cite{CC1} and Theorem 1.7 in \cite{CC1}, we have $\frac{1}{C}\omega_g\leq\omega_{f}\leq C\omega_g$ for some constant $C$. Since $u=f^*$, we get the bound of all higher derivatives of $u$. By Arzela-Ascoli theorem, we have $[0,1]\subset \Lambda$.

\section{\bf Scalar curvatures on homogeneous toric bundles}\label{Sec-5}

In this section, we generalize our results to homogeneous toric bundles. First, we briefly review homogeneous toric bundles and the generalized Abreu equation, For the details, see \cite{D5} and \cite{CHLLS}. Then we sketch the proof of Theorem \ref{theorem_1.1}.

\subsection{Homogeneous Toric Bundles}\label{toric-fibration}

Let $G$ be a compact semisimple Lie group, $K$ be the centralizer of a torus $S$ in $G$,
and $T$ be a maximal torus in $G$ containing $S$.
Then, $T\subset C(S)=K$ and $G/K$ is a generalized flag manifold.

Let $Z(K)$ be the center of $K$, which is an $n$-dimensional torus, denoted by $T^{n}$.
Let $(M,\omega)$ be a compact toric K\"{a}hler manifold of complex dimension $n$, where $T^{n}$
acts effectively on $M$. Let $\varrho: K\rightarrow T^{n}$ be a surjective homomorphism.
The homogeneous toric bundle $G\times_{K}M$ is defined to be the space $G\times M$ modulo the relation
$$(gh,x)=(g,\varrho(h)x)\quad\text{for any }g\in G, h\in K, x\in M.$$
Later on, we will omit $\varrho$ to simplify notations.
The space $G\times_{K}M$ is a fiber bundle with fiber $M$ and base space $G/K$,
a generalized flag manifold. There is a natural $G$-action on $G\times_{K}M$   given by
$$g\cdot[h,x]=[gh,x]
\quad\text{for any }g\in G, x\in M,$$
and a natural $T^{n}$-action on $G\times_{K}M$ given by
$$k\cdot[h,x]=[h, k^{-1}x]\quad\text{for any }k\in T^{n}.$$

\subsection{The Generalized Abreu Equation}\label{sec-2.2}

Denote by $\tau:M\rightarrow \bar{\Delta}\subset \mathfrak{t}^*$
the moment map of $M$, where $\Delta$ is a Delzant polytope. We extend $\tau: G\times_{K}M\rightarrow \bar{\Delta}$ by $\tau([g,x])=\tau(x)$.

Suppose that $\Delta$ is defined by linear inequalities $l_i(\xi):=\langle\xi,a_i\rangle- \lambda_i>0$,  for $i=1, \cdots, d$. Set
$
v(\xi)=\sum_i l_i(\xi)\log l_i(\xi).
$
It defines a K\"ahler metric on $G\times_{K}M$, which we also call the {Guillemin} metric.
For any strictly convex function $u$ with $u-v\in C^{\infty}(\bar{\Delta})$, there is a $(G,T^n)$-invariant K\"ahler metric $\mathcal{G}_{u}$ on $G\times_{K}M$ with scalar curvature given by
\begin{equation}\label{eqn 4.6}
\mathcal{S}= -\frac{1}{\mathbb {D}}\sum_{i,j=1}^n\frac{\partial^2(\mathbb {D}u^{ij})}{\partial \xi_i\partial \xi_j} + h_G.
\end{equation}
Here, $\mathbb{D}$ is called the Duistermaat-Heckman polynomial and $h_G$ is a known function.

For any strictly convex function $u$ with $u-v\in C^{\infty}(\bar{\Delta})$, there exist $(1,0)$-vector fields
$$S_{j}:=\frac{1}{2}\left(\frac{\partial}{\partial x_i}-\sqrt{-1}H_{i}\right), (i=1,...,n)$$
$$S_{\alpha}:=\frac{1}{2}\left(V_{\alpha}-\sqrt{-1}W_{\alpha}\right), \alpha\in R_{M}^{+}.$$
where the left-invariant vector fields $\{\frac{\partial}{\partial x_i}, H_i, V_{\alpha}, W_{\alpha}\}$ form a local basis of $G\times_{K}M$. The vector fields $H_i, V_{\alpha}, W_{\alpha}$ and the index $R_M^+$ are determined by $G$ and $K$. For more details, see \cite{CHLLS}, \cite{LSZ}. Let $\{dx_i, \nu_{i}, dV^{\alpha}, dW^{\alpha}\}_{1\leq i\leq n,\alpha\in R_{M}^{+}}$ be the dual 1-form of the basis.

The $(G,T^n)$-invariant K\"ahler metric $\mathcal{G}_{u}$ on $G\times_{K}M$ is given by
$$\mathcal{G}_{u}=\sum_{i,j=1}^{n}(f_{ij}dx_i\otimes dx_j+f_{ij}\nu^i\otimes\nu^j)+\sum_{\alpha\in R_M^+}D_{\alpha}(dV^{\alpha}\otimes dV^{\alpha}+dW^{\alpha}\otimes dW^{\alpha}).$$

Let $h$ be a $(G,T^n)$-invariant function defined on $G\times_{K}M$ which depends only on $\xi$. If the vector field $grad^{1,0}h$ is holomorphic, we have
$$\sum_{i,j=1}^{n}\frac{\partial}{\partial x_j}\left(f^{ij}\frac{\partial h}{\partial x_j}\right)=\sum_{i,j=1}^{n} u^{ij}\frac{\partial^2 h}{\partial \xi_i\partial \xi_j}=0.$$
So $h$ is a linear function.

In particular, if the metric is cscK, $\mathcal{S}$ is a constant function. If the metric is extremal, $\mathcal{S}$ is a linear function. Set $A=\mathcal{S}-h_G$.

We will consider the equation
\begin{equation}\label{gAbreu}
-\frac{1}{\mathbb {D}}\sum_{i,j=1}^n\frac{\partial^2(\mathbb {D}u^{ij})}{\partial \xi_i\partial \xi_j}=A,
\end{equation}
where $\mathbb{D}>0$ and $A$ are given functions defined on $\bar{\Delta}$.
The equation \eqref{gAbreu} was introduced by Donaldson \cite{D5}
in the study of scalar curvatures of toric fibrations. See also \cite{R} and \cite{N-1}.
We call \eqref{gAbreu} the generalized Abreu Equation. Note that even if $\mathcal{S}$ is a constant function, the right hand of equation (\ref{gAbreu}) is not a constant function.

\subsection{$K$-stabilities for homogeneous toric bundles}

Following Donaldson \cite{D1} (also see \cite{N-1}, \cite{CHLLS}), for any smooth function $A$ on $\bar\Delta$, we can define a functional on $\mc C$:
$$\mc F_{A}(u)= -\int_{\Delta}\log \det(u_{ij})\mathbb {D}d\mu + \mc L_{A}(u),$$
where $\mc L_{A}$ is the linear functional
$$\mc L_{A}(u) = \int_{\partial \Delta}u \mathbb {D}d\sigma - \int_{\Delta}Au
\mathbb {D}d\mu,$$
where $d\mu$ is the  Lebesgue measure on $\mathbb R^n$, and on each face $F$, $d\sigma$ is a constant multiple of the standard $(n-1)$-dimensional Lebesgue measure (see \cite{D1} for details).

The equivariant test configurations for homogeneous toric bundles are encoded by rational piecewise-linear convex functions on $\bar\Delta$. See \cite{N-1} for toric fibrations and \cite{Del} for spherical varieties. Repeating the argument of Szi\'ekelyhidi, one can check that any convex function on $\bar{\Delta}$ gives rise to a filtration of the homogeneous coordinate ring. So we have the following definitions.

\begin{defn}\label{definition_2.1.1}
Let $A\in
C^{\infty}(\bar\Delta)$ be a smooth function on $\bar\Delta$.
$({\Delta},\mathbb D, A)$ is called {\it relatively $K$-polystable} if $\mc
L_{A}(u)\geq 0$ for all rational piecewise-linear convex functions
$u$, and $\mc
L_{A}(u)=0$ if and only if $u$ is a linear function.
\end{defn}

\begin{defn}\label{definition_2.1.5}
$({\Delta},\mathbb{D},A)$ is called {\it uniformly relatively $K$-polystable}
if for any $u\in {\mc P}_{p_o}(\Delta)$
$$
\mc L_A(u)\geq \lambda\int_{\partial \Delta} u \mathbb{D}d \sigma
$$
for some constant $\lambda>0$. Sometimes, we say that $\Delta$ is
$(\mathbb{D},A,\lambda)$-stable.
\end{defn}

\begin{defn}\label{defn_toric_hat_k-1}
Let $A\in C^{\infty}(\bar\Delta)$ be a smooth function on $\bar\Delta$. $({\Delta},\mathbb D, A)$ is called {\it relatively $\widehat{K}$-polystable} if $\mc L_{A}(u)\geq 0$ for all convex  functions $u:\bar \Delta\to \mathbb R$, and $\mc L_{A}(u)=0$ if and only if $u$ is a linear function.
\end{defn}

When $A+h_G$ is a constant function, the Definition \ref{defn_toric_hat_k-1} is the same as the definition for the homogeneous toric bundles in \cite{Sz1}.

\v\v
\subsection{\bf Proof of Theorem \ref{theorem_1.1}}

Since $\mathbb D$ is positive and bounded,  we can prove the following theorem as Theorem \ref{theorem_stab} with a weight function.

\begin{theorem}\label{theorem_stab-1}
Assume that $(\Delta,\mathbb D,A)$ is relatively $\widehat K$-polystable. Let $u\in\mathcal C_{*}^{P}$, such that $u$ is lower semi-continuous, $\int_{\partial \Delta }u\mathbb D d\sigma =1, u(o)=\inf_{\bar{\Delta}}u=0$ and  $\mathcal L_{A}(u)=0$. Then
\begin{equation}
u\in L^{\infty}(\bar{\Delta}).
\end{equation}
\end{theorem}

By Theorem \ref{theorem_stab-1}, we have the equivalence of relative $\widehat{K}$-stability and uniformly relative $K$-stability for homogeneous toric bundles. The proof is the same as in Theorem \ref{theorem_5.2}.

\begin{theorem}\label{theorem_5.2-1}
Let $(M,\omega)$ be a $n$-dimensional compact toric manifold and $\Delta$ be its Delzant polytope. Let $G/K$ be a generalized flag manifold and $G\times_KM$ be the homogeneous toric bundle. Let $\mathbb{D}>0$ be the Duistermaat-Heckman polynomial, and $A\in C^{\infty}(\bar{\Delta})$ be a given smooth function. Then, $(\Delta, \mathbb{D}, A)$ is relatively $\widehat K$-polystable if and only if there is $\lambda>0$ such that $(\Delta,\mathbb{D},A)$ is uniformly relatively $K$-polystable.
\end{theorem}

By Theorem \ref{theorem_5.2-1}, we need to prove the necessity and sufficiency of uniformly relative $K$-polystability. By Theorem 2.3 in \cite{CHLLS} and Lemma \ref{lemma_1.2.12},  uniformly relatively $K$-polystability is a necessary condition. So we only need to prove the sufficiency of uniformly relatively $K$-polystability.

\v
For homogeneous toric bundles, using the bound of $\mathbb D$, by the same argument as in Section \ref{Sec-4}, we have the following theorem
\begin{theorem}\label{theorem 5.1-1}
If $\Delta$ is $(\mathbb{D}, A,\lambda)$-stable, then the entropy is bounded.
\end{theorem}
Then using continuity method, we can prove Theorem \ref{theorem_1.1}.

\section{\bf Optimal Degenerations}\label{Sec_6}
In this section we consider $K$-unstable toric variety for constant scalar curvatures, the argument can be generalized to the general case. For any convex function $u$, define
$$\|u\|^2=\int_{\Delta}u^{2}d\mu-\frac{\left(\int_{\Delta}ud\mu\right)^2}{Vol(\Delta,d\mu)}.$$
$\|u\|$ is also a norm for the subspace of convex functions with $\int_\Delta ud\mu=0.$

Consider the following functional
\begin{equation}
W(u)=\frac{\mathcal L_{a}(u)}{\|u\|},\;\;\;\;
\end{equation}
where $a=\frac {Vol(\partial \Delta,d\sigma )}{Vol(\Delta,d\mu)},$  $0\neq u\in \mathcal{C}_1\cap L^2(\Delta)$. Here $\mathcal C_1$ is the set of continuous convex functions on $\Delta^*$, integrable on $\partial\Delta$, where $\Delta^*$ is the union of $\Delta$ and the interiors of its co-dimension one faces. Sz\'ekelyhidi has proved that there exists a convex minimizer $u\in  \mathcal{C}_1\cap L^2(\Delta)$ with $\int_{\Delta}u d\mu=0$, which is unique up to scaling (see \cite{Sz2}).  We will prove that this $u$ is bounded.

\begin{remark}
Note that Sz\'ekelyhidi uses $L^2$-norm for the subspaces of convex functions with $\int_\Delta ud\mu=0.$  Since $\|\cdot\|$ and the functional $W(u)$ are invariant when we replace $u$  by $u+c,$ where $c$ is a constant, all results are the same.
\end{remark}

\begin{theorem}
Assume that the toric variety with moment polytope $\Delta $ be unstable. Then there exists a convex minimizer $u\in L^\infty(\bar \Delta)$ for $W$, which gives a filtration.  $u$ is unique up to scaling.
\end{theorem}
The proof is similar to the proof of Theorem \ref{theorem_stab}.
\begin{proof}
Uniqueness is given by Theorem 5 in \cite{Sz2}. We only need to prove that $|u|_{L^\infty(\Delta)}<+\infty.$ Let $o$ be the center of $\Delta$. Choose coordinate system $\xi_i$ such that $\xi(o)=0.$

By the result of Sz\'ekelyhidi, we can assume that
$$
\int_{\Delta}ud\mu =0,\qquad\int_{\partial \Delta}ud\mu\leq C_o,\qquad \|u\|_{L^2(\Delta)}=1.
$$
Then $\|u\|=1.$ Since $u$ is continuous in $\Delta,$ we have
$$|u(o)|\leq C_1, \quad |p^o|\leq C_1,\quad\forall p^o\in Du(o)$$
for some constant $C_1>0.$

We can assume that $u$ is lower semi-continuous. In fact, let $u^*$ be the lower semi-continuous regularization of $u.$ Since $u$ is continuous and convex, $u|_{\Delta }=u^*|_{\Delta }$ and $u|_{\partial \Delta }\geq u^{*}|_{\partial \Delta }$, we have
$$\mathcal L_a(u^*)\leq\mathcal L_a(u),\;\;\;\;\|u^*\|=\|u\|.$$
So $ W(u)\leq W(u^*)\leq W(u)$, $u^*$ is also a minimizer.

Since $\mathcal L_a(u)$ is invariant by adding a linear function, $W(u)=\inf_{\mathcal C_1\cap L^2(\Delta)} W,$ we get,
$$
\|u\|\leq  \|u+\ell(\xi)\|
$$
for any linear function $\ell(\xi).$ Set $u^o=u-u(o)-\langle p^o,\xi\rangle,$ where $p^o\in Du(o).$ Then $u^o$ is normalized at $o$. Note that $u^o$ may not be minimizer. Let $f^o$ be the Legendre transform of $u^o.$ As in the proof of  Theorem \ref{theorem_stab} we can repeat the argument for $u^o$ and construct  $u^o_h$. Since $o\in W_h,$ $u^o_h$ is also normalized at $o.$ Set $u_h=u^o_h+u(o)+\langle p^o,\xi \rangle.$ Then we have
\begin{equation}
0\leq u_h^o \leq u^o,\quad u_h\leq u,\quad u_h|_{W_h}=u|_{W_h},\quad u-u_h\equiv u^o-u^o_h.
\end{equation}
Set $v_h=u-u_h.$ Since $\mathcal L_a$ is invariant by adding a linear function, we have
\begin{equation}
\mathcal L_a(u)=\mathcal L_a(u^o),\quad\mathcal L_a(u_h)=\mathcal L_a(u^o_h).
\end{equation}
We only need  the following claim, the other argument is word by word.

{\bf Claim.} Assume that  $\bar\Delta\setminus D f^o(\overline{S_{f^o}(0,h)})\neq \emptyset$ for any $h>0.$ Then
$$W(u) >W(u_h),$$
as $h$ large enough.
Set $\ell(\xi)=u(o)+\langle p^o,\xi\rangle .$
Since $0\leq u^o_h\leq u^o$ and $u=u^o+\ell,u_h=u_h^o+\ell$.
\begin{align*}
\|u_h\|^2&=\|u\|^2+\int_\Delta\left(-(u^o+\ell)^2+(u^o_h+\ell)^2\right)d\mu +\frac{\left(\int_\Delta ud\mu\right)^2-\left(\int_\Delta u_hd\mu\right)^2}{Vol(\Delta,d\mu)} \\
&=\|u\|^2+\int_\Delta \left((u_h^o)^2-(u^o)^2\right)d\mu -2\int_{\Delta} \ell v_hd\mu  +\frac{\left(\int_\Delta v_hd\mu\right)\left(\int_\Delta ( 2u_h+v_h)d\mu\right)}{Vol(\Delta,d\mu)} \\
&\leq \|u\|^2+ 2\max_{\bar\Delta}|\ell|\int_{\Delta}v_h d\mu+\frac{\left(\int_\Delta v_hd\mu\right)^2 }{Vol(\Delta,d\mu)}\\
&= \left(\|u\|+\frac{\max_{\bar\Delta}|\ell|\int_{\Delta}v_h d\mu}{\|u\|}\right)^2+\left(\frac{1}{Vol(\Delta,d\mu)}-\frac{\max_{\bar\Delta}|\ell|^2}{\|u\|^2}\right)\left(\int_{\Delta}v_h d\mu\right)^2
\end{align*}
where we used
$$
\int_\Delta\left( (u^o)^2-(u^o_h)^2\right)d\mu \geq 0,\quad v_h\geq 0,\quad  \int_\Delta u_hd\mu \leq \int_\Delta ud\mu=0$$
in the last inequality.
Hence by $\|u\|=1$ and the Cauchy-Schwarz inequality   we have
$$
\|u_h\|\leq \|u\|+C_2\int_{\Delta} v_hd\mu,
$$
where $C_2= \max_{\bar\Delta}|\ell|+\frac{1}{\sqrt{Vol(\Delta,d\mu)}}.$ 
Since $\int_{\partial \Delta }u^od\mu \leq C_o, |u(o)|\leq C_1, |p^o|\leq C_1,$ we have 
 $$0\leq \int_{\partial \Delta } v_h d\mu \leq \int_{\partial \Delta }u^o d\mu \leq \int_{\partial \Delta }ud\mu+ \int_{\partial \Delta }|\ell|d\mu\leq C.$$
Combining the estimate \eqref{est_int} and \eqref{est_bou}, as in the proof of Theorem \ref{theorem_stab}, we have
$$
\int_\Delta v_hd\mu \leq  C_3\max_{e\in S^{n-1}(1)}(R_e-r_e)\int_{\partial \Delta } v_h d\sigma,
$$
and
$$\mathcal L_a(u_h)<\mathcal L_a(u)< 0$$
where $C_3=\max_i  |n_{F_i}|.$ Applying the estimates above we obtain that
\begin{align*}
W(u_h)=\frac{\mathcal L_a(u)-\mathcal L_a(v_h)}{\|u_h\|}\leq \frac{\mathcal L_a(u)-(1-aC_3\max_{e\in S^{n-1}(1)}(R_e-r_e))\int_{\partial \Delta } v_h d\sigma}{\|u\|+C_2C_3\max_{e\in S^{n-1}(1)}(R_e-r_e)\int_{\partial \Delta } v_h d\sigma}.
\end{align*}
Hence when $ C_3(a-W(u)C_2)(\max_{e\in S^{n-1}(1)}(R_e-r_e)) \leq \frac{1}{2},$ we have
\begin{align*}
&W(u_h)-W(u) \\
\leq &\frac{\left[-W(u)C_2C_3\max_{e\in S^{n-1}(1)}(R_e-r_e)-(1-aC_3\max_{e\in S^{n-1}(1)}(R_e-r_e))\right]\int_{\partial \Delta } v_h d\sigma}{\|u\|+C_2C_3\max_{e\in S^{n-1}(1)}(R_e-r_e)\int_{\partial \Delta } v_h d\sigma} \\
\leq & -\frac{\int_{\partial \Delta } v_h d\sigma}{2\left(\|u\|+C_2C_3\max_{e\in S^{n-1}(1)}(R_e-r_e)\int_{\partial \Delta } v_h d\sigma\right) }<0.
\end{align*}
The claim is proved.

Since $W(u)=\inf_{\mathcal C_1\cap L^2(\Delta)} W$, we can conclude that there exists a constant $h>0$ such that  $\bar\Delta\setminus D f^o(\overline{S_{f^o}(0,h)})=\emptyset$. Then as in the proof of Theorem  \ref{theorem_stab} we have
$|u^o|_{L^\infty(\Delta)} \leq C_4,$ and hence
$$
|u |_{L^\infty(\Delta)} \leq  |u^o|_{L^\infty(\Delta)}+\max_{\bar\Delta}|\ell|\leq C_5,
$$
for some constants $C_4,C_5$. The theorem is proved.
\end{proof}

\bibliography{<your-bib-database>}

\end{document}